\documentclass[12pt]{amsart}

\usepackage{amsmath, amssymb, graphics}

\newcommand{\mathsym}[1]{{}}
\newcommand{\unicode}[1]{{}}
\newtheorem{thm}{Theorem}[section]

\newtheorem{lem}[thm]{Lemma}
\newtheorem{prop}[thm]{Proposition}
\theoremstyle{definition}

\newtheorem{defn}[thm]{Definition}
\newtheorem{rem}[thm]{Remark}

\newtheorem*{defn*}{Definition}
\newtheorem*{rems*}{Remarks}
\newtheorem*{rem*}{Remark}
\newtheorem{ex}[thm]{Example}

\numberwithin{equation}{section}

\begin{document}

\title[Singular symplectic forms] {On local invariants of singular symplectic forms. }
\author{Wojciech Domitrz}
\address{Warsaw University of Technology\\ Faculty of Mathematics and Information
Science\\
ul. Koszykowa 75\\ 00-662 Warszawa\\ Poland}
\email{domitrz@mini.pw.edu.pl}
\thanks{The research was supported by NCN grant no. DEC-2013/11/B/ST1/03080.}
\date{}

\maketitle

\begin{abstract}
We find a complete set of local invariants of singular symplectic forms with the structurally stable Martinet hypersurface on a $2n$-dimensional  manifold. In the $\mathbb C$-analytic category this set consists of the Martinet hypersurface $\Sigma _2$, the restriction of the singular symplectic form $\omega$ to $T\Sigma_2$ and the kernel of $\omega^{n-1}$ at the point $p\in \Sigma_2$.  In the $\mathbb R$-analytic and smooth categories this set contains one more invariant: the canonical orientation of $\Sigma_2$. We find the conditions to determine  the kernel of $\omega^{n-1}$ at $p$ by the other invariants. In dimension $4$ we find sufficient conditions to determine the equivalence class of a singular symplectic form-germ with the structurally smooth Martinet hypersurface by the Martinet hypersurface and the restriction of the singular symplectic form to it.
We also study the singular symplectic forms with singular Martinet hypersurfaces. We prove that the equivalence class of such singular symplectic form-germ is determined by the Martinet hypersurface, the canonical orientation of its regular part and the restriction of the singular symplectic form to its regular part if the Martinet hypersurface is a quasi-homogeneous hypersurface with an isolated singularity.
\end{abstract}

\section{Introduction.}
A closed differential $2$-form $\omega$ on a $2n$-dimensional
smooth manifold $M$ is {\bf symplectic} if $\omega$ is
nondegenerate. This means that $\omega$ satisfies the following
condition
\begin{equation}\label{Darboux}
\omega^n|_p=\omega\wedge\cdots\wedge\omega|_p\ne 0, \ \text {for}
\ p\in M.
\end{equation}

A closed differential $2$-form $\omega$ on a $2n$-dimensional
smooth manifold $M$ is called a {\bf singular symplectic} form if
the set of points where $\omega$ does not satisfy (\ref{Darboux}):
\begin{equation}\label{Martinet}
\left\{p\in M:\omega^n|_p=0\right\}.
\end{equation}
is nowhere dense. We denote the set (\ref{Martinet}) by
$\Sigma_2(\omega)$ or $\Sigma_2$. It is called the {\bf Martinet
hypersurface}.

Singular symplectic forms appear naturally
 if one studies classification of germs of submanifolds of a symplectic manifold.
 By Darboux-Givental theorem (\cite{ArGi}, see also \cite{DJZ2})  germs of submanifolds of the symplectic manifold
 are symplectomorphic iff the restrictions of the symplectic form to them are diffeomorphic.
 This theorem reduces the problem of local classification of generic submanifolds of the symplectic manifold to the problem of
 local classification of singular symplectic forms.
 
Singular symplectic forms can be applied in thermodynamics: in the modeling the absolute zero temperature region (see \cite{Ja2}). The Martinet $\Sigma_{20}$ singular symplectic form gives a fine link between the thermodynamical postulate of positivity of absolute temperature and the stability of an applicable structure of thermodynamics  (\cite{JK}).

By the classical Darboux theorem all symplectic forms on $M$ are
locally diffeomorphic i.e. there exists a diffeomorphism-germ of
$M$ mapping the germ of one symplectic form to the germ of the
other.

This is no longer true if we consider singular symplectic forms.
It is obvious that if germs of singular symplectic forms
$\omega_1$ and $\omega_2$ are diffeomorphic then the germs of
corresponding Martinet hypersurfaces $\Sigma_2(\omega_1)$ and
$\Sigma_2(\omega_2)$ must be diffeomorphic and the restrictions of
germs of singular symplectic forms $\omega_1$ and $\omega_2$ to
the regular parts of $\Sigma_2(\omega_1)$ and $\Sigma_2(\omega_2)$
respectively must be diffeomorphic too.

In this paper we study if the inverse theorem is valid:

{\it Do the Martinet hypersurface $\Sigma_2$ and the restriction
of $\omega$ to the regular part of $\Sigma_2$ form a complete set
of invariants of $\omega$?}

Because our consideration is local, we may assume that $\omega$ is
 a $\mathbb K$-analytic or smooth closed $2$-form-germ at $0$ on
$\mathbb K^{2n}$ for $\mathbb K= \mathbb R$ or $\mathbb K= \mathbb
C$.

Then $\omega^n=f\Omega$, where $f$ is a function-germ at $0$ and
$\Omega$ is the germ at $0$ of a volume form on $\mathbb K^{2n}$.
The Martinet hypersurface has the form
$\Sigma_2=\left\{f=0\right\}$ and it is a called {\bf structurally
smooth at $0$} if $f(0)=0$ and $df_0 \ne 0$. Then $\Sigma_2$ is
 a smooth hypersurface-germ. In dimension $4$ such situation
is generic.

The starting point of this paper is the articles  \cite{JZ1} and
\cite{JZ2} where the similar problems where concerned for singular
contact structures. B. Jakubczyk and M. Zhitomirskii show that
local $\mathbb C$-analytic singular contact structures on $\mathbb
C^3$ with structurally smooth Martinet hypersurfaces $S$ are
diffeomorphic if their Martinet hypersurfaces and restrictions of
singular structures to them are diffeomorphic. In the $\mathbb
R$-analytic category a complete set of invariants contains, in
general, one more independent invariant. It is  a canonical
orientation on the Martinet hypersurface. The same is true for
smooth local  singular contact structures $P=(\alpha)$ on $\mathbb
R^3$ provided $\alpha|_S$ is either not flat at $0$ or
$\alpha|_S=0$. The authors also study local singular contact
structures in higher dimensions. They find more subtle invariants
of a singular contact structure $P=(\alpha)$ on $\mathbb K^{2n+1}$
: a line bundle $L$ over the Martinet hypersurface $S$, a
canonical partial connection $\Delta_0$ on the line bundle $L$ at
$0\in \mathbb K^{2n+1}$  and a $2$-dimensional kernel
$ker(\alpha\wedge(d\alpha)^n)|_0$. They also consider the more
general case when $S$ has singularities.

For the first occurring singularities of singular symplectic forms
on a $4$-dimensional manifold the answer for the above question
follows from Martinet's normal forms ( see \cite{Martinet},
\cite{Roussarie}, \cite{G_T} ). In fact it is proved that the
Martinet hypersurface $\Sigma_2$ and a characteristic line field
on $\Sigma_2$ (i.e. $\{X \ \text{is a smooth vector field} :
X\rfloor (\omega|_{T\Sigma_2})=0\}$) form a complete set of
invariants of generic singularities of singular symplectic forms on a $4$-dimensional manifold.

In this paper we show that a complete set of invariants for 
$\mathbb C$-analytic singular symplectic form-germs on $\mathbb C^{2n}$
with structurally smooth Martinet hypersurfaces consists of the
Martinet hypersurface, the pullback of the singular form-germ $\omega$ to it and
the $2$-dimensional kernel of $\omega^{n-1}|_0$ (Theorem
\ref{inv-C}). The same is true for local $\mathbb R$-analytic
and smooth singular symplectic forms on $\mathbb R^{2n}$ with
structurally smooth Martinet hypersurfaces if we include in the set of
invariants the canonical orientation of the Martinet hypersurface
(Theorem \ref{inv-R}).

In section \ref{deter} we also prove that an equivalence class of
a smooth or $\mathbb K$-analytic singular symplectic form-germ $\omega$
on $\mathbb K^{2n}$ with the structurally smooth Martinet
hypersurface is determined only by the Martinet hypersurface, its
canonical orientation ( only if $\mathbb K=\mathbb R$ ) and the
pullback of the singular form-germ to it if the dimension of a vector space generated by the coefficients of the $1$-jet at $0$ of $(\omega|_{T\Sigma_2})^{n-1}$ is equal to $2$.

In section \ref{deter4} we consider singular symplectic forms on
$\mathbb K^4$ with structurally smooth Martinet hypersurfaces.
 We
show that an equivalence class of a smooth or $\mathbb K$-analytic singular
symplectic form $\omega$ on $\mathbb K^4$ with a structurally
smooth Martinet hypersurface is determined only by the Martinet
hypersurface and the pullback of the singular form to it if the two generators of the ideal generated by  coefficients of $\omega|_{T\Sigma_2}$ form a regular sequence.

In $\mathbb C$-analytic category we prove the same result for a wider class of singular symplectic forms.
The analogous result in $\mathbb R$-analytic category requires the assumption on the canonical orientation.
The preliminary versions of results of section \ref{deter4} were presented in \cite{4dim}
(Theorems \ref{C-ana}, \ref{R-ana}, Proposition \ref{propC}).

We also consider singular symplectic forms with singular Martinet
hypersurfaces. We prove that if the Martinet hypersurface of a
singular symplectic form-germ is  a quasi-homogeneous
hypersurface-germ with an isolated singularity then the complete set of
local invariants of this singular form consists of the canonical
orientation of the regular part of the Martinet hypersurface (for
$\mathbb K=\mathbb R$ only) and the restriction of the singular
form to the regular part of the Martinet hypersurface.

{\bf Acknowledgement.} The author wishes to express his thanks to
B. Jakubczyk and M. Zhitomirskii for many helpful
conversations and remarks during writing this paper.

\section{The complete set of invariants for singular symplectic
forms with structurally smooth Martinet hypersurfaces.}

\subsection{The kernel of $\omega^{n-1}|_0$}

The kernel of $\omega^{n-1}|_0$ is the following $2$-dimensional
subspace of $T_0\mathbb K^{2n}$
$$
\ker \left(\omega^{n-1}|_0\right)=\{v\in T_0\mathbb K^{2n}: v\rfloor
\left(\omega^{n-1}|_0\right)=0\}
$$
The kernel $\ker \left(\omega^{n-1}|_0\right)$ can be also described as a kernel of a
$(2n-3)$-form on $\Sigma_2$. Let $Y$ be a vector field-germ on
$\mathbb K^{2n}$ that is transversal to $\Sigma_2$ at $0$. Let
$\iota :\Sigma_2\hookrightarrow \mathbb K^{2n}$ be the inclusion.
Then the kernel of $\iota^{\ast}(Y\rfloor\omega^{n-1})|_0$ is
equal to $\ker \omega^{n-1}|_0$.

\subsection{The canonical orientation of $\Sigma_2$}

 In $\mathbb R$-analytic and smooth categories there is
one more invariant in general. This is a {\it canonical
orientation of $\Sigma_2$}.  The orientation may be defined
invariantly. Let $\omega$ be a singular symplectic
form-germ on $\mathbb R^{2n}$ with a structurally smooth Martinet
hypersurface $\Sigma_2$ at $0$. Then $\Sigma_2=\{f=0\}$ and
$df|_0\ne 0$. We define the volume form $\Omega_{\Sigma_2}$ on
$\Sigma_2$ which determines the orientation of $\Sigma_2$ in the
following way
$$
df\wedge\Omega_{\Sigma_2} =\frac{\omega^n}{f}.
$$
If $f$ is singular at $0$  (see Section \ref{singular}) then we  define the canonical orientation on the regular part of $\Sigma_2=\{f=0\}$

This definition is analogous to the definition in \cite{JZ1}
proposed by V. I. Arnold. It is easy to see that this definition
of the orientation does not depend on the choice of $f$ such that
$\Sigma_2=\{f=0\}$ and $df|_0\ne 0$.
  We call this
orientation of $\Sigma_2$ the {\it canonical orientation of
$\Sigma_2$}.

\begin{ex}
Let $\omega_{0}$, $\omega_{1}$ be germs of the following singular
symplectic forms on $\mathbb K^4$
$$
\omega_{0}=d(p_1(dx-zdy))+xdx\wedge dy, \
\omega_{1}=d(p_1(dy+zdx))+xdx\wedge dy
$$
in the coordinate system  $(p_1,x,y,z)$ on $\mathbb K^4$.

It is easy to see that $\omega_0^2=\omega_1^2=2p_1dp_1\wedge
dx\wedge dy\wedge dz$. Thus
$\Sigma_2=\Sigma_2(\omega_0)=\Sigma_2(\omega_1)=\{p_1=0\}$,
$\sigma=\iota^{\ast}\omega_0=\iota^{\ast}\omega_1=xdx\wedge dy$
and the canonical orientations of $\Sigma_2$ are the same for
$\omega_0$ and $\omega_1$.

But the kernels of $\omega_0|_0$ and $\omega_1|_0$ are different.
One can check that
$$
\ker
(\omega_0|_0)=\ker (dp_1\wedge dx)|_0=\text{span}\{\frac{\partial}{\partial
y}|_0,\frac{\partial}{\partial z}|_0\}
$$
and
$$
\ker
(\omega_1|_0)=\ker (dp_1\wedge dy)|_0=\text{span}\{\frac{\partial}{\partial
x}|_0,\frac{\partial}{\partial z}|_0\}
$$

Let $\Sigma_{22}=\{(x,y,z)\in\Sigma_2: \sigma_{(x,y,z)}=0\}$. It
is easy to see that $\Sigma_{22}=\{(x,y,z)\in\Sigma_2: x=0\}$.

Then $\ker (\omega_0|_0)$ is tangent to
$\Sigma_{22}$ and $\ker (\omega_1|_0)$ is
transversal to $\Sigma_{22}$. Therefore $\omega_0$ and $\omega_1$
are not equivalent.
\end{ex}
\subsection{Main theorems for  structurally smooth Martinet
hypersurfaces.}

In the $\mathbb C$-analytic category $\omega$ is determined by the
restriction to $T\Sigma_2$ and the $2$-dimensional
kernel of $\omega^{n-1}|_0$.

\begin{thm}
\label{inv-C} Let $\omega_0$ and $\omega_1$ be germs of  $\mathbb
C$-analytic singular symplectic forms on $\mathbb C^{2n}$ with a
common structurally smooth Martinet hypersurface $\Sigma_2$ at $0$
and $\text{rank}(\iota^{\ast}\omega_0|_{0})=\text{rank}(\iota^{\ast}\omega_1|_0) \le
2n-4$.

If $\iota^{\ast}\omega_0=\iota^{\ast}\omega_1$ and $\ker \omega_0^{n-1}|_0=
\ker\omega_1^{n-1}|_0$ then  there exists  a $\mathbb C$-analytic
diffeomorphism-germ $\Psi :(\mathbb C^{2n},0)\rightarrow (\mathbb
C^{2n},0)$ such that
\[
\Psi ^{\ast }\omega_1 =\omega_0.
\]
\end{thm}

In $\mathbb R$-analytic and smooth categories $\omega$ is
determined by the restriction to $T\Sigma_2$, the $2$-dimensional
kernel of $\omega^{n-1}|_0$ and the canonical
orientation of $\Sigma_2$.

\begin{thm}
\label{inv-R} Let $\omega_0$ and $\omega_1$ be germs of  smooth
($\mathbb R$-analytic) singular symplectic forms on $\mathbb
R^{2n}$ with a common structurally smooth Martinet hypersurface
$\Sigma_2$ at $0$  and
$\text{rank}(\iota^{\ast}\omega_0|_{0})=\text{rank}(\iota^{\ast}\omega_1|_0) \le
2n-2$.

If the canonical orientations defined by $\omega_0$ and $\omega_1$ are the same, $\iota^{\ast}\omega_0=\iota^{\ast}\omega_1$ and $\ker \omega_0^{n-1}|_0=
\ker\omega_1^{n-1}|_0$ then
there exists a smooth ($\mathbb R$-analytic)
diffeomorphism-germ $\Psi :(\mathbb R^{2n},0)\rightarrow (\mathbb
R^{2n},0)$ such that
\[
\Psi ^{\ast }\omega_1 =\omega_0.
\]
\end{thm}

Theorems \ref{inv-C} and \ref{inv-R} are corollaries of Theorem
\ref{4-dim}. Proofs of Theorems \ref{inv-C} and \ref{inv-R} are
presented in the next section.

\section{A normal form and a realization theorem for singular
symplectic forms
with structurally smooth Martinet hypersurfaces.}

The main result of this section is Theorem \ref{4-dim}. In this
theorem a 'normal' form of $\omega$ with the given pullback to the
Martinet hypersurface is presented and a sufficient conditions for
equivalence of germs of singular symplectic forms with the same
pullback to the common Martinet hypersurface are found. We also
show which germs of closed $2$-forms on $\mathbb K^{2n-1}$ may be
obtained as a pullback to a structurally smooth Martinet
hypersurface of a singular symplectic form-germ on $\mathbb
K^{2n}$.   All results of this section hold in $\mathbb
C$-analytic, $\mathbb R$-analytic and ($C^{\infty}$) smooth
categories.

Let $\Omega$ be  a volume form-germ on $\mathbb K^{2n}$. Let
$\omega_0$ and $\omega_1$ be two germs of singular symplectic
forms on $\mathbb K^{2n}$ with structurally smooth Martinet
hypersurfaces at $0$. It is obvious that if there exists a
diffeomorphism-germ of $\mathbb K^{2n}$ at $0$  such that
$\Phi^{\ast}\omega_1=\omega_0$ then
$\Phi(\Sigma_2(\omega_0))=\Sigma_2(\omega_1)$. Therefore we assume
that these singular symplectic forms have the same Martinet
hypersurface.

If the singular symplectic form-germs are equal on their common
Martinet hypersurface than we obtain the following result ( see
{\cite{Geometry} ).
\begin{prop}
\label{rel-Darboux}
 Let $\omega _{0}$ and $\omega _{1}$ be two germs at $0$ of
singular symplectic forms on $\mathbb K^{2n}$ with the common
structurally smooth Martinet hypersurface $\Sigma_2$.

If $\frac{\omega _{1}^{n}}{\omega _{0}^{n}}|_0>0$ and $\omega
_{0}|_{T_{\Sigma_2}{\mathbb K^{2n}}}=\omega
_{1}|_{T_{\Sigma_2}{\mathbb K^{2n}}}=\tilde{\omega}$ then there
exists a diffeomorphism-germ $\Phi :(\mathbb K^{2n},0)\rightarrow
(\mathbb K^{2n},0)$ such that
\[
\Phi ^{\ast }\omega _{1}=\omega _{0}
\]
and $\Phi|_{\Sigma_2}=Id_{\Sigma_2}$.
\end{prop}

\begin{rem}
The assumption $\frac{\omega _{1}^{n}}{\omega _{0}^{n}}|_0>0$ is
needed only in $\mathbb R$-analytic and smooth categories. In the
$\mathbb C$-analytic category we may assume that $\Re
e\left(\frac{\omega _{1}^{n}}{\omega _{0}^{n}}|_0\right)>0$ or
$\Im m\left(\frac{\omega _{1}^{n}}{\omega
_{0}^{n}}|_0\right)\ne0$. But this is a technical assumption (see
Remark \ref{C-orient}).
\end{rem}

\begin{proof} We present the proof in $\mathbb R$-analytic and smooth categories.
The proof in the $\mathbb C$-analytic category is similar. Firstly
we simplify the form-germs $\omega _{0}$ and $\omega _{1}$. We find the
local coordinate system  such that $\omega _{0}^{n}=p_1 \Omega $, $
\omega _{1}^{n}=p_1 (A+g)\Omega $, where $\Omega=dp_1\wedge dq_1
\wedge \cdots\wedge dp_n \wedge dq_n$, $g$ is a function-germ,
$g(0)=0$ and $A>0$ (see \cite{Martinet}). By assumptions, we have
$\omega _{i}=p_{1}\alpha _{i}+\tilde{\omega}$, where $\alpha_i$
and $\tilde{\omega}$ are germs of $2$-forms and
$\tilde{\omega}|_{T_{\left\{ p_{1}=0\right\} }\mathbb
R^{2n}}=\omega _{i}|_{T_{\left\{ p_{1}=0\right\} }\mathbb R^{2n}}$
for $i=0,1$. Then further on we use the Moser homotopy method (see
\cite{Moser}). Let $\omega _{t}=t\omega _{1}+(1-t)\omega _{0}$,
for $t\in \lbrack 0;1]$.

We want to find a family of diffeomorphisms $\Phi_t$, $t\in [0;1]$
such that $\Phi_t^{\ast}\omega_t=\omega_0,$ for $t \in [0;1]$, \
$\Phi_0=Id.$ Differentiating the above homotopy equation by $t$,
we obtain
\[
d(V_t\rfloor\omega_t)=\omega_0-\omega_1=p_1(\alpha_0-\alpha_1),
\]
where $V_t=\frac{d}{dt}{\Phi_t}$. Now we prove the following
lemma.

\begin{lem}
\label{Poincare} If $p_{1}\alpha $ is  a closed $2$-form-germ
on $\mathbb R^{2n}$ then there exists  a $1$-form-germ $\beta
$ such that $p_{1}\alpha =d(p_{1}^{2}\beta )$. \
\end{lem}

\begin{proof}[Proof of Lemma \ref{Poincare}.] By the Relative Poincare Lemma (see
\cite{ArGi}, \cite{DJZ1}) there
exists  a $1$-form-germ $\gamma $ such that $p_{1}\alpha
=d(p_{1}\gamma )=dp_{1}\wedge \gamma +p_{1}d\gamma $. Therefore $dp_{1}\wedge \gamma |_{T_{\left\{
p_{1}=0\right\} }\mathbb R^{2n}}=0$. Hence there exist a
$1$-form-germ $\delta $ and a smooth function-germ $f$
such that $\gamma =p_{1}\delta +fdp_{1}$. If we take $\beta =\delta -\frac{df}{2}$ then
\[
p_{1}\alpha =d(p_{1}\gamma -d(\frac{p_{1}^{2}f}{2}))=d(p_{1}^{2}\beta ),
\]
which finishes the proof of Lemma \ref{Poincare}.
\end{proof}

The $2$-form  $p_{1}(\alpha _{0}-\alpha _{1})=\omega _{1}-\omega _{0}$
is closed. By the above lemma we have
\begin{equation}
V_{t}\rfloor \omega _{t}=p_{1}^{2}\beta .  \label{linear}
\end{equation}

Now we calculate $\Sigma_2(\omega_t)$. It is easy to see that
\[
\omega _{i}^{n}=(p_{1}\alpha _{i}+\tilde{\omega})^{n}=\tilde{\omega}%
^{n}+p_{1}\sum_{k=1}^{n}\left( _{k}^{n}\right) p_{1}^{k-1}\alpha
_{i}^{k}\wedge \tilde{\omega}^{n-k}.
\]
But $\omega _{i}^{n}|_{T_{\left\{ p_{1}=0\right\} }R^{2n}}=0$.
This clearly forces $\tilde{\omega}^{n}=0$. By the above formula
we get
\[
n\alpha _{0}\wedge \tilde{\omega}^{n-1}=\Omega
-p_{1}\sum_{k=2}^{n}\left( _{k}^{n}\right) p_{1}^{k-2}\alpha
_{0}^{k}\wedge \tilde{\omega}^{n-k}
\]
and
\[
n\alpha _{1}\wedge \tilde{\omega}^{n-1}=(A+g)\Omega
-p_{1}\sum_{k=2}^{n}\left( _{k}^{n}\right) p_{1}^{k-2}\alpha
_{1}^{k}\wedge \tilde{\omega}^{n-k}
\]

The above formulas imply the following formula
\begin{eqnarray}
\omega _{t}^{n} &=&(p_{1}(t\alpha _{1}+(1-t)\alpha
_{0})+\tilde{\omega})^{n}=
 \nonumber \\
&=&p_{1}(tn\alpha _{1}\wedge \tilde{\omega}^{n-1}+(1-t)n\alpha
_{0}\wedge
\tilde{\omega}^{n-1})+  \nonumber \\
&&+\sum_{k=2}^{n}\left( _{k}^{n}\right) p_{1}^{k}(t\alpha
_{1}+(1-t)\alpha
_{0})^{k}\wedge \tilde{\omega}^{n-k}  \nonumber \\
&=&p_{1}(1+t(A+g-1))\Omega + \label{det}   \\
&&+p_{1}^{2}\sum_{k=2}^{n}\left( _{k}^{n}\right) p_{1}^{k-2}\left(
(t\alpha _{1}+(1-t)\alpha _{0})^{k}-t\alpha _{1}^{k}-(1-t)\alpha
_{0}^{k}\right) \wedge \tilde{\omega}^{n-k} .\nonumber
\end{eqnarray}
From (\ref{det}) we obtain
\begin{equation}
\omega_t^n=p_{1}(1+t(A+g-1)+p_{1}h_{t})\Omega,\label{n-linear}
\end{equation}
where $h_{t}$ is a function-germ. But
$(1+t(A-1))>0 $ for $A>0$ and $t\in [0,1]$.

$\Sigma_2(\omega_t)=\left\{p_1=0\right\}$ is nowhere dense,
therefore by direct algebraic calculation, it is easy to see that
equation (\ref{linear}) is equivalent
 to the following equation
\begin{equation}\label{2l}
V_{t}\rfloor \omega _{t}^n=n p_{1}^{2}\beta \wedge \omega_t^{n-1}.
\end{equation}

Combining (\ref{2l}) with (\ref{n-linear}) we obtain
\begin{equation}
V_{t}\rfloor (1+t(A+g-1)+p_{1}h_{t})\Omega=n p_{1}\beta \wedge
\omega_t^{n-1}. \label{non-linear}
\end{equation}
But if $A>0$ then $(1+t(A-1))>0 $ for $t\in[0;1]$.
 Therefore we can find a smooth (or $\mathbb R$-analytic) vector field-germ $V_{t}$ that
satisfies (\ref{non-linear}). The restriction of $V_{t}$ to $\Sigma_2$ vanishes, because the
right hand side of (\ref{non-linear}) vanishes on $\Sigma_2$.
Hence there exists a diffeomorphism $\Phi _{t}$ such that $\Phi
_{t}^{\ast }\omega _{t}=\omega _{0}$ for $t\in \lbrack 0,1]$ and
$\Phi_t|_{\Sigma_2}=Id_{\Sigma_2}$. This completes the proof of
Theorem \ref{rel-Darboux}.
\end{proof}

If $\text{rank} (\iota^{\ast}\omega|_0)$ is $2n-2$ then $\omega$ is
equivalent to $\Sigma_{20}$ Martinet's singular form (see
\cite{Martinet}). Therefore we study singular symplectic forms
such that $\text{rank}( \iota^{\ast}\omega|_0)\le 2n-4$. In fact we will
prove that structural stability of $\Sigma_2(\omega)$ implies
that $\text{rank} (\iota^{\ast}\omega|_0)=2n-4$

In the next theorem we describe all germs of singular symplectic
forms $\omega$ on $\mathbb K^{2n}$ with structurally smooth
Martinet hypersurfaces at $0$ and $\text{rank} (\iota^{\ast}\omega|_0)\le
2n-4$. We also find the sufficient conditions for equivalence of
singular symplectic forms of this type. This is a generalisation of the analogous result for singular symplectic forms on $4$-dimensional manifolds (\cite{4dim}).

 We use the following  mappings in the subsequent results $\iota :\Sigma_2=\left\{ p_{1}=0\right\}\hookrightarrow \mathbb K^{2n}$
\[
\iota (p_{2},\cdots,p_{n},q_{1},\cdots,q_{n})=
(0,p_{2},\cdots,p_{n},q_{1},\cdots,q_{n})
\]
and $\pi :\mathbb K^{2n}\rightarrow \Sigma_2=\left\{
p_{1}=0\right\}$
\[
\pi(p_{1},p_{2},\cdots,p_{n},q_{1},\cdots,q_{n})=
(p_{2},\cdots,p_{n},q_{1},\cdots,q_{n})  .
\]

\begin{thm}
\label{4-dim} Let $\omega$ be  a singular symplectic form-germ
on $\mathbb K^{2n}$ with a structurally smooth Martinet
hypersurface at $0$.

(a) If $\text{rank}(\iota^{\ast}\omega|_{0}) \le 2n-4$ then there exists a diffeomorphism-germ $\Phi :(\mathbb K^{2n},0)\rightarrow
(\mathbb K^{2n},0)$ such that
\[
\Phi ^{\ast }\omega =d\left( p_{1}\pi ^{\ast }\alpha \right) +\pi
^{\ast }\sigma ,
\]

where $\sigma =\iota^{\ast}\Phi^{\ast}\omega$ is  a
closed $2$-form-germ on $\left\{ p_{1}=0\right\} $  and $\alpha $ is  a $1$-form-germ on $\left\{ p_{1}=0\right\}$ such that
$\alpha \wedge \sigma^{n-1} =0$ and $\alpha\wedge d\alpha \wedge
\sigma^{n-2}|_0\ne 0$.

(b) Moreover if $\omega_0=d\left( p_{1}\pi ^{\ast }\alpha_0 \right)
+\pi ^{\ast }\sigma$ and $\omega_1=d\left( p_{1}\pi ^{\ast
}\alpha_1 \right) +\pi ^{\ast }\sigma$ are two germs of singular
symplectic forms satisfying the above conditions  and
\begin{enumerate}
\item  \label{orientation1} $\frac{\alpha_1 \wedge d\alpha_1\wedge
\sigma^{n-2} }{ \alpha_0 \wedge d\alpha_0\wedge
\sigma^{n-2}}|_{0}>0$, (only for $\mathbb K=\mathbb R$)

\item \label{lineincot} $\alpha_1|_0 \wedge \alpha_0|_0 \wedge
\sigma^{n-2}|_0=0$,
\end{enumerate}
then there exists a diffeomorphism-germ $\Psi :(\mathbb
K^{2n},0)\rightarrow (\mathbb K^{2n},0)$  such that
\[
\Psi ^{\ast }\omega_1 =\omega_0.
\]
\end{thm}

\begin{rem}\label{C-orient}
Assumption (\ref{orientation1}) is only needed in $\mathbb
R$-analytic and smooth categories. In the $\mathbb C$-analytic
category we have
$$
\Phi^{\ast}(d\left( p_{1}\pi ^{\ast }\alpha \right) +\pi ^{\ast
}\sigma)=d\left( p_{1}\pi ^{\ast }i\alpha \right) +\pi ^{\ast
}\sigma,
$$
 where $\Phi$
is the following diffeomorphism
$$
\Phi(p_1,p_{2},\cdots,p_{n},q_{1},\cdots,q_{n}))=(ip_1,p_{2},\cdots,p_{n},q_{1},\cdots,q_{n})
$$
and $i^2=-1$. It is obvious that $\Phi|_{\Sigma_2}=Id_{\Sigma_2}$,
where $\Sigma_2=\{p_1=0\}$ and $i\alpha\wedge d(i\alpha)\wedge
\sigma^{n-2}=-\alpha\wedge d\alpha \wedge \sigma^{n-2}$.
\end{rem}

\begin{proof} The proof is similar to the proof of analogous theorem for singular symplectic forms on a $4$-dimensional manifold (see \cite{4dim}).  We can find a coordinate system  $(p_1, q_1,\cdots, p_n, q_n)$ such that $\Sigma_2(\omega)=\{p_1=0\}$. Then by the Relative Poincare Lemma (see \cite{ArGi}, \cite{DJZ1}) there exists $1$-form-germ $\gamma$ on $\mathbb K^{2n}$ such that
$
\omega=d(p_1\gamma)+\pi^{\ast}\sigma$. It is clear that we can write $\gamma $ in the following form
$\gamma=\pi^{\ast}\alpha+p_1\delta+g dp_1$, where $\alpha$ is a $1$-form-germ on $\left\{ p_1=0 \right\}$, $g$ is a
function-germ
 and $\delta$
is a $1$-form-germ on $\mathbb K^{2n}$. Then
\[
d(p_1(p_1\delta+g dp_1))=p_1(2 dp_1 \wedge \delta+p_1 d\delta +dg \wedge
dp_1).
\]
By Lemma \ref{Poincare} we have
$\omega=d(p_1\pi^{\ast}\alpha)+\pi^{\ast}\sigma+d(p_1^2 \theta)$.

Hence
\begin{eqnarray*}
\omega ^{n}&=&n dp_1 \wedge \pi^{\star}\alpha \wedge
\pi^{\star}(\sigma^{n-1})+2np_{1}dp_{1}\wedge \pi ^{\star }\beta
\wedge \pi ^{\star
}(\sigma ^{n-1})\\
& &+n(n-1)p_{1}dp_{1}\wedge \pi ^{\star }\alpha \wedge d\pi
^{\star }\alpha \wedge \pi ^{\star }(\sigma
^{n-2}))+p_{1}^{2}v\Omega ,
\end{eqnarray*}
where $v$ is a function-germ  at $0$ on $\mathbb K^{2n}$.  We have $\alpha \wedge
\sigma^{n-1}=0$, because $\omega ^{n}|_{T_{\left\{ p_{1}=0\right\}
}\mathbb K^{2n}}=0$. From $\sigma ^{n-1}|_{0}=0$, we have
\[
\omega ^{n}=n(n-1)p_{1}dp_{1}\wedge \pi ^{\star }\alpha \wedge
d\pi ^{\star }\alpha \wedge \pi ^{\star }(\sigma
^{n-2})+p_{1}g\Omega.
\]
where $g$ is a function-germ on $\mathbb K^{2n}$ vanishing at $0$.  From the above we
obtain that
$$
\alpha \wedge d\alpha \wedge \sigma ^{n-2}|_0\ne 0.
$$
Therefore
\begin{equation} \label{rank}
\text{rank}(\sigma|_0)=2n-4.
\end{equation}

Let
$$
\omega _{0}=d\left( p_{1}\pi ^{\ast }\alpha \right) +\pi ^{\ast
}\sigma.
$$
Then
\[
\omega _{0}^{n}=n(n-1)p_{1}dp_{1}\wedge \pi ^{\star }\alpha \wedge
d\pi ^{\star }\alpha \wedge \pi ^{\star }(\sigma
^{n-2})+p_{1}h\Omega ,
\]
where $h$ is a function-germ on $\mathbb K^{2n}$ vanishing at $0$.
One can check that
\[
\tilde{\omega}=\omega _{0}|_{T_{\left\{ p_{1}=0\right\} }\mathbb
K^{2n}}=dp_{1}\wedge \pi ^{\star }\alpha +\pi ^{\star }\sigma
=\omega |_{T_{\left\{ p_{1}=0\right\} }\mathbb K^{2n}}.
\]
Therefore by Proposition \ref{rel-Darboux} there exists a germ of
a diffeomorphism $\Theta :(\mathbb K^{2n},0)\rightarrow (\mathbb
K^{2n},0)$ such that $\Theta ^{\ast }\omega =\omega _{0}$ and
$\Theta|_{\left\{p_1=0\right\}}= Id_{\left\{p_1=0\right\}}$.

This finish the proof of part (a)

Now we prove part (b). (\ref{rank}) and (\ref{lineincot}) implies
that there exists $B\ne 0$ such that
$\alpha_1|_0\wedge\sigma^{n-2}|_0=B
\alpha_0|_0\wedge\sigma^{n-2}|_0$. If $B\ne 1$ then
$\Phi^{\ast}\omega_1=d(p_1\pi^{\ast}(B\alpha))+\pi^{\ast}\sigma$
where  $\Phi$ is a diffeomorphism-germ of the form $\Phi(p,q)=(B
p_1,p_2,...,p_n,q_1,...,q_n)$). Thus we may assume that $B=1$.

We use the Moser homotopy method. Let
$\alpha_t=t\alpha_1+(1-t)\alpha_0$ and $\omega_t=d\left( p_{1}\pi
^{\ast }\alpha_t \right) +\pi ^{\ast }\sigma$ for $t\in [0,1]$. It
is easy to check that $\alpha_t \wedge \sigma^{n-1}=0$.

Now we look for germs of diffeomorphisms $\Phi_t$ such that
\begin{equation}
\label{hom} \Phi_t^{\ast}\omega_t=\omega_0,\  \text{for} \ t \in [0;1], \
\Phi_0=Id.
\end{equation}
 Differentiating the above homotopy equation
by $t$, we obtain
\[
d(V_t\rfloor\omega_t)=d(p_1\pi^{\ast}(\alpha_0-\alpha_1)),
\]
where $V_t=\frac{d}{dt}{\Phi_t}$. Therefore we have to solve the
following equation
\begin{equation}\label{rl}
V_t\rfloor\omega_t=p_1\pi^{\ast}(\alpha_0-\alpha_1).
\end{equation}

We calculate the Martinet hypersurface of $\omega_t$.
$$
\omega_t^n=n(n-1)p_1dp_1 \wedge \pi^{\star}(\alpha_t \wedge
d\alpha_t \wedge \sigma ^{n-2})+p_{1}^2g_t\Omega,
$$
 where $g_t$
is a smooth function-germ at $0$, because $\sigma^{n}=0$,
$(d\alpha_t)\wedge\sigma^{n-1}=0$ and
$\alpha_t\wedge\sigma^{n-1}=0$.

Now we calculate
\begin{eqnarray*}
&\alpha_t \wedge d\alpha_t \wedge \sigma ^{n-2}|_{0}=& \nonumber \\
&=t^2\alpha_1 \wedge d\alpha_1 \wedge \sigma
^{n-2}|_{0}+t(1-t)\alpha_1 \wedge d\alpha_0 \wedge
\sigma ^{n-2}|_{0}+& \nonumber \\
&+t(1-t)\alpha_0 \wedge d\alpha_1 \wedge \sigma ^{n-2}|_{0}+
(1-t)^2\alpha_0 \wedge d\alpha_0 \wedge \sigma ^{n-2}|_{0}&. \nonumber \\
\end{eqnarray*}
From $\alpha_0 \wedge \sigma ^{n-2}|_0= \alpha_1 \wedge \sigma
^{n-2}|_0$ we have
\begin{eqnarray*}
&\alpha_t \wedge d\alpha_t \wedge \sigma ^{n-2}|_{0}=& \nonumber \\
&=(t^2+t(1-t))d\alpha_1 \wedge \alpha_1 \wedge \sigma ^{n-2}|_{0}& \nonumber \\
&+(t(1-t)+(1-t)^2) d\alpha_0 \wedge \alpha_0 \wedge \sigma ^{n-2}|_{0}=& \nonumber \\
&=t \alpha_1 \wedge d\alpha_1 \wedge \sigma ^{n-2}|_{0}+
(1-t) \alpha_0 \wedge d\alpha_0 \wedge \sigma ^{n-2}|_{0}&. \nonumber \\
\end{eqnarray*}
But there exists $A>0$ such that $\alpha_1 \wedge d\alpha_1 \wedge
\sigma ^{n-2}|_{0}= A \alpha_0 \wedge d\alpha_0 \wedge \sigma
^{n-2}|_{0}$, so we obtain
\begin{eqnarray*}
&\alpha_t \wedge d\alpha_t \wedge \sigma ^{n-2}|_{0}=& \nonumber \\
&=(At+(1-t))\alpha_0 \wedge d\alpha_0 \wedge \sigma ^{n-2}|_{0}\ne 0\nonumber \\
\end{eqnarray*}
for $t\in [0,1]$. Therefore
\begin{equation*}
dp_1\wedge\pi^{\ast}(\alpha_t\wedge d\alpha_t\wedge \sigma
^{n-2})|_0 \ne 0
\end{equation*}
for $t\in[0;1]$. Thus $\Sigma_2(\omega_t)=\{p_1=0\}$.

 Because $\Sigma_2$ is nowhere dense, equation (\ref{rl}) is equivalent to
\[
V_t\rfloor\omega_t^n=np_1\pi^{\star}(\alpha_0-\alpha_1)\wedge\omega_t^{n-1}
\]
and $\omega_t^n=n(n-1)p_1dp_1 \wedge \pi^{\star}(\alpha_t \wedge
d\alpha_t \wedge \sigma ^{n-2})+p_{1}^2g_t\Omega$ , where $g_t$ is
a smooth function-germ at $0$. Hence we have to solve the
following equation
\begin{equation}
\label{n2-linear} V_t\rfloor \left(n(n-1)dp_1 \wedge
\pi^{\star}(\alpha_t \wedge d\alpha_t \wedge \sigma
^{n-2})+p_1g_t\Omega
\right)=n\pi^{\star}(\alpha_0-\alpha_1)\wedge\omega_t^{n-1}.
\end{equation}
 From the above calculation we have $\alpha_t \wedge d\alpha_t
\wedge \sigma ^{n-2}|_{0}\ne 0$. Therefore $n(n-1)dp_1 \wedge
\pi^{\star}(\alpha_t \wedge d\alpha_t \wedge \sigma ^{n-2})+
p_1g_t\Omega$ is a nondegenerate $2n$-form-germ  on $\mathbb
K^{2n}$ and
\begin{eqnarray*}
&n\pi^{\star}(\alpha_0-\alpha_1)\wedge\omega_t^{n-1}|_0=&\nonumber \\
&n(n-1)dp_1\wedge\pi^{\star}(\alpha_1\wedge\alpha_0\wedge\sigma^{n-2})|_0=0,& \nonumber \\
\end{eqnarray*}
because $\alpha_1 \wedge \alpha_0 \wedge \sigma ^{n-2}|_0=0$. Hence
we can find a smooth solution $V_t$ of (\ref{n2-linear}) such that
$V_t|_0=0$. Thus there exit germs of diffeomorphisms $\Phi_t$,
which satisfy (\ref{hom}). For $t=1$ we have
$\Phi_1^{\star}\omega_1=\omega_0$.
\end{proof}

Now we can prove main theorems from the previous section.
\begin{proof}[Proof of Theorems \ref{inv-C} and \ref{inv-R}]
It is easy to see that if
$\omega=d(p_1\pi^{\ast}\alpha)+\pi\sigma$, where $\alpha$ and
$\sigma$ satisfy conditions of Theorem \ref{4-dim}  then
$\ker \omega^{n-1}|_0=\ker(\alpha\wedge\sigma^{n-2})|_0$ and the canonical orientation
of $\Sigma_2$ is defined by the volume form $\alpha\wedge d\alpha
\wedge \sigma^{n-2}$. By Theorem \ref{4-dim} we get the result.
\end{proof}

We call a closed $2$-form-germ $\sigma$ on $\mathbb K^{2n-1}$
{\it realizable with a structurally smooth Martinet hypersurface}
if there exists  a singular symplectic form-germ $\omega$ on
$\mathbb K^{2n}$ such that $\Sigma_2(\omega)= \{0\}\times \mathbb
K^{2n-1}$ is structurally smooth and
$\omega|_{T\Sigma_2(\omega)}=\sigma$.

From Martinet's normal form of  a singular symplectic
form-germ on $\mathbb K^{2n}$ of the rank $2n-2$ we know that all germs
of closed $2$-forms on $\mathbb K^{2n-1}$ of the rank $2n-2$ are
realizable with a structurally smooth Martinet hypersurface.  From
part (a) of the Theorem \ref{4-dim} we obtain the following
realization theorem of closed $2$-forms on $\mathbb K^{2n-1}$ of
the rank less than $2n-2$ at $0\in \mathbb K^{2n-1}$.

\begin{thm}
Let $\sigma$ be a closed $2$-form-germ on $\mathbb K^{2n-1}$
and $\text{rank}( \sigma|_0)<2n-2$.  $\sigma$ is realizable with a
structurally smooth Martinet hypersurface if and only if $\text{rank}
\sigma|_0=2n-4$ and there exists a $1$ form-germ $\alpha$ on $\mathbb
K^{2n-1}$ such that $\alpha\wedge \sigma^{n-1}=0$ and
$\alpha\wedge d\alpha \wedge \sigma^{n-2}|_0\ne 0$.

\end{thm}


\section{Determination by the restriction of $\omega$ to $T\Sigma_2$ and
the canonical orientation of $\Sigma_2$.}\label{deter}

In this section we find sufficient conditions to determine the equivalence class of a singular symplectic form by its restriction to the structurally smooth Martinet hypersurface $\Sigma_2$ and the canonical orientation of $\Sigma_2$.

Let $j^1_0f$ denote the $1$-jet at $0$ of a smooth ($\mathbb K$-analytic) function-germ $f:\mathbb K^{2n-1} \rightarrow \mathbb K$. The vector space of all $1$-jets at $0$ of smooth $\mathbb K$-analytic) function-germs on $\mathbb K^{2n-1}$ is denoted by $J^1_0(\mathbb K^{2n-1},\mathbb K)$.

Let $\sigma$ be a closed $2$-form-germ at $0$ on $\mathbb K^{2n-1}$. Then  the closed $(2n-2)$-form-germ $\sigma^{n-1}$ at $0$ on $\mathbb K^{2n-1}$ has the following form in a local coordinates set $q=(q_1,\cdots,q_{2n-1})$ on $\mathbb K^{2n-1}$
$$\sigma^{n-1}=\sum_{i=1}^{2n-1} g_i dq_1\wedge\cdots \wedge dq_{i-1}\wedge dq_{i+1}\wedge\cdots\wedge dq_{2n-1},$$
where $g_i:\mathbb K^{2n-1}\rightarrow \mathbb K$ is a smooth ($\mathbb K$-analytic) function-germ at $0$ for $i=1,\cdots,2n-1$.

 Hence the $1$-jet at $0$ of $2n-2$-form-germ $\sigma^{n-1}$ has the following form
 $$j^1_0\sigma^{n-1}=\sum_{i=1}^{2n-1} j^1_0g_idq_1\wedge\cdots \wedge dq_{i-1}\wedge dq_{i+1}\wedge\cdots\wedge dq_{2n-1}.$$
  We denote by $\text{span} j^1_0\sigma^{n-1}$ the vector space spanned by coefficients of $j^1_0\sigma^{n-1}$
  $$
  \text{span} j^1_0\sigma^{n-1}=\text{span}\left(j^1_0g_1,\cdots,j^1_0g_{2n-1}\right).
  $$
If $g_i(0)=0$ then $j^1_0g_i=\sum_{k=1}^{2n-1}\frac{\partial g_i}{\partial q_k}(0)q_k$. Thus it easy to check that if $\text{rank}(\sigma|_0)=2n-4$ then the definition of $\text{span}j^1_0\sigma^{n-1}$ does not depend on the choice of a local coordinate system .

\begin{thm} \label{tw-deter}
Let $\omega_0$ and $\omega_1$ be germs of smooth ($\mathbb K$-analytic)
singular symplectic forms on $\mathbb K^{2n}$ with a common
structurally smooth Martinet hypersurface $\Sigma_2$ at $0$ and
$\text{rank}(\iota^{\ast}\omega_0|_{0})=\text{rank}(\iota^{\ast}\omega_1|_0) =2n-4$.

If $\iota^{\ast}\omega_0=\iota^{\ast}\omega_1=\sigma$, $\omega_0$
and $\omega_1$ define the same canonical orientation of $\Sigma_2$
and the dimension of the vector space $\text{span}j_0^1\sigma^{n-1}$ is $2$ then there exists a smooth ($\mathbb K$-analytic)
diffeomorphism-germ  $\Psi :(\mathbb K^{2n},0)\rightarrow (\mathbb
K^{2n},0)$ such that
\[
\Psi ^{\ast }\omega_1 =\omega_0.
\]
\end{thm}

The proof is based on the following lemma.
\begin{lem}\label{lemat}

Let $\sigma$ be a closed smooth ($\mathbb K$-analytic) $2$-form-germ at $0$ on $\mathbb K^{2n-1}$ such that $\text{rank} (\sigma|_0)=2n-4$. Let $\alpha_0$, $\alpha_1$ be smooth ($\mathbb K$-analytic) $1$-form-germs at $0$ on $\mathbb K^{2n-1}$ such that for $i=0,1$
\begin{equation}\label{adas}
\alpha_i\wedge d\alpha_i\wedge \sigma^{n-2}|_0\ne 0
 \end{equation}
\begin{equation}\label{as}
\alpha_i\wedge \sigma^{n-1}= 0
\end{equation}
If the dimension of a vector space $\text{span} \ j^1_0\sigma^{n-1}$ is $2$  then
there exists a number $A\ne 0$ such that $\alpha_0\wedge \sigma^{n-2}|_0=A \alpha_1\wedge \sigma^{n-2}|_0$.
\end{lem}

\begin{proof}[Proof of Lemma \ref{lemat}]  Since $\text{rank}(\sigma|_0)=2n-4$, there exists a local coordinate system  $(x_1,\cdots,x_{2n-4},y_1,y_2,y_3)$ on $\mathbb K^{2n-1}$ and function-germs $a_i, b_{ij}, c_{ij}$ on $\mathbb K^{2n-1}$ vanishing at $0$ such that
\begin{eqnarray}\label{sigma}
\sigma=\sum_{k=1}^{n-2}dx_{2k-1}\wedge dx_{2k}+\sum_{1\le i < j\le 2n-4} c_{ij} dx_i\wedge dx_j \\ +\sum_{i=1}^3\sum_{j=1}^{2n-4} b_{ij} dy_i\wedge dx_j+ \sum_{\{i,j,k\}=\{1,2,3\} \ j<k}a_i dy_j\wedge dy_k. \nonumber
\end{eqnarray}

It implies that the $1$-jet of $\sigma^{n-1}$ at $0$ has the following form
\begin{equation}\label{1-jet}
j^1_0 \sigma^{n-1}= \sum_{\{i,j,k\}=\{1,2,3\} \ j<k}j^1_0a_i dy_j\wedge dy_k \wedge dx_1 \wedge \cdots \wedge dx_{2n-4},
\end{equation}
where $j^1_0a_i$ denotes the $1$-jet of the function-germ $a_i$ at $0$ for $i=1,2,3$.

The vector space $\text{span}j^1_0 \sigma^{n-1}$ is spanned by $j^1_0a_1, j^1_0a_2, j^1_0a_3$.

There exist function-germs $f_{ij}$ and $g_{ik}$ for $i=0,1$, $j=1,2,3$, $k=1,\cdots,2n-4$ such that
$$\alpha_i=\sum_{j=1}^3 f_{ij}dy_j+\sum_{k=1}^{2n-4} g_{ik}dx_k.$$
By (\ref{adas}) we get that $f_{01}\ne0$ or $f_{02}\ne 0$ or $f_{03}\ne 0$. Without loss of generality we may assume that $f_{03}\ne 0$, since we can change a coordinate system  replacing $y_j$ with $y_3$ if $f_{03}=0$ and $f_{0j}\ne 0$ for $j\ne 3$.

By (\ref{as}) we get $j^1_0(\alpha_0\wedge \sigma^{n-1})=0$. By (\ref{1-jet}) it implies that $$f_{01}(0)j_0^1a_1+f_{02}(0)j_0^1a_2+f_{03}(0)j_0^1a_3=0,$$
 since $a_i(0)=0$ for $i=1,2,3$.
Since $f_{03}(0)\ne 0$ we get that
\begin{equation}\label{a3}
j_0^1 a_3=-\frac{f_{01}(0)}{f_{03}(0)}j_0^1a_1-\frac{f_{02}(0)}{f_{03}(0)}j_0^1a_2.
\end{equation}
Thus  the space $\text{span} \ j^1_0 \ \sigma^{n-1}$ is spanned by $j^1_0a_1, j^1_0a_2$. Since $\dim \text{span} \ j^1_0 \sigma^{n-1}=2$ the 1-jets $j^1_0a_1, j^1_0a_2$ are $\mathbb K$-linearly independent.
On the other hand by  (\ref{as}) we get $j^1_0(\alpha_1\wedge \sigma^{n-1})=0$. By (\ref{1-jet}) it implies that $$f_{11}(0)j_0^1a_1+f_{12}(0)j_0^1a_2+f_{13}(0)j_0^1a_3=0,$$
since $a_i(0)=0$ for $i=1,2,3$.
By (\ref{a3}) it implies that
$$\left(f_{11}(0)-\frac{f_{13}(0)}{f_{03}(0)}f_{01}(0)\right)j_0^1a_1+\left(f_{12}(0)-\frac{f_{13}(0)}{f_{03}(0)}f_{02}(0)\right)j_0^1a_2=0.$$
Since the 1-jets $j^1_0a_1, j^1_0a_2$ are $\mathbb K$-linearly independent we get that
\begin{equation}\label{final}
f_{11}(0)-\frac{f_{13}(0)}{f_{03}(0)}f_{01}(0)=f_{12}(0)-\frac{f_{13}(0)}{f_{03}(0)}f_{02}(0)=0.
\end{equation}
By (\ref{sigma}) we get that $\sigma^{n-2}|_0=(n-2)!dx_1\wedge\cdots \wedge dx_{2n-4}|_0$. Thus we have for $i=0,1$ $$\alpha_i\wedge \sigma^{n-2}|_0=(n-2)!\sum_{j=1}^3 f_{ij}(0) dy_i\wedge dx_1 \wedge \cdots \wedge dx_{2n-4}|_0.$$
By (\ref{final}) it implies that
$\alpha_1\wedge \sigma^{n-2}|_0=\frac{f_{13}(0)}{f_{03}(0)}\alpha_0\wedge \sigma^{n-2}|_0$.
\end{proof}

\begin{proof}[Proof of Theorem \ref{tw-deter}]
By Theorem \ref{4-dim} we can find a local coordinate system  such that the germs $\omega_0$ and $\omega_1$ have the following form
 $\omega_0=d\left( p_{1}\pi ^{\ast }\alpha_0 \right)
+\pi ^{\ast }\sigma$ and $\omega_1=d\left( p_{1}\pi ^{\ast
}\alpha_1 \right) +\pi ^{\ast }\sigma$, where $\alpha_0, \alpha_1, \sigma$ are form-germs satisfying the assumptions of Lemma \ref{lemat}. Thus there exists a number $A\ne 0$ such that $\alpha_0\wedge \sigma^{n-2}|_0=A \alpha_1\wedge \sigma^{n-2}|_0$. By Theorem \ref{4-dim} it implies that there exists a smooth ($\mathbb K$-analytic)
diffeomorphism-germ  $\Psi :(\mathbb K^{2n},0)\rightarrow (\mathbb
K^{2n},0)$ such that
\[
\Psi ^{\ast }\omega_1 =\omega_0.
\]
\end{proof}

\begin{ex}

Let $\omega$ be the following closed $2$-form-germ on $\mathbb K^{2n}$
\begin{eqnarray}
&\omega=d(p_1(dy_3+y_1dy_2)) +\sum_{k=1}^{n-2}dx_{2k-1}\wedge dx_{2k}+& \label{2-form}  \\
&(dy_3+y_1dy_2)\wedge(b(y_1,y_2,y_3)dy_1-a(y_1,y_2,y_3)dy_2)& \nonumber
\end{eqnarray}
where $(p_1,y_1,y_2,y_3,x_1,\cdots,x_{2n-4})$ is a coordinate system  on $\mathbb K^{2n}$, $b$ is a smooth ($\mathbb K$-analytic) function-germ on $\mathbb K^3$ vanishing at $0$, $h$ is a smooth ($\mathbb K$-analytic) function-germ on $\mathbb K^2$ vanishing at $0$ and
\begin{equation}\label{a-formula}
a(y_1,y_2,y_3)=\int_{0}^{y_1}\left(t\frac{\partial b}{\partial y_3}(t,y_2,y_3)-\frac{\partial b}{\partial y_2}(t,y_2,y_3)\right)dt +h(y_2,y_3).
\end{equation}

It is easy to see that the Martinet hypersurface is $\Sigma_2=\{p_1=0\}$
and the restriction of $\omega$ to $T\Sigma_2$ has the following form
$$
\sigma=(dy_3+y_1dy_2)\wedge(b(y_1,y_2,y_3)dy_1-a(y_1,y_2,y_3)dy_2)
+\sum_{k=1}^{n-2}dx_{2k-1}\wedge dx_{2k}.
$$

Thus $j_0^1\sigma^{n-1}$ is equal to
$$(n-2)!\left((j_0^1b) dy_3 \wedge dy_1 + (j_0^1a) dy_2 \wedge dy_3\right) \wedge dx_1\wedge\cdots \wedge dx_{2n-4}. $$
Then the space $\text{span} \  j_0^1\sigma^{n-1}$ is $\text{span} \left\{j_0^1a, j_0^1b\right\}$. From (\ref{a-formula}) we get
$$a(0)=0, \ \ \frac{\partial a}{\partial y_1}(0)=-\frac{\partial b}{\partial y_2}(0), \ \ \frac{\partial a}{\partial y_i}(0)=\frac{\partial h}{\partial y_i}(0) \ \text{for} \ i=2,3.$$

Hence  $\text{span} \  j_0^1\sigma^{n-1}$ is spanned by
$$-\frac{\partial b}{\partial y_2}(0)y_1+\frac{\partial h}{\partial y_2}(0)y_2+\frac{\partial h}{\partial y_3}(0)y_3, \ \ \frac{\partial b}{\partial y_1}(0)y_1+\frac{\partial b}{\partial y_2}(0)y_2+\frac{\partial b}{\partial y_3}(0)y_3.$$

Thus $\dim \text{span} \  j_0^1\sigma^{n-1}$ is $2$ if and only if the rank of the following matrix is $2$.

$$\left[\begin{array}{ccc}
-\frac{\partial b}{\partial y_2}(0)&\frac{\partial h}{\partial y_2}(0)&\frac{\partial h}{\partial y_3}(0)\\
\frac{\partial b}{\partial y_1}(0)&\frac{\partial b}{\partial y_2}(0)&\frac{\partial b}{\partial y_3}(0)
\end{array}\right]$$

For $n=2$ any closed $2$-form-germ  satisfying the assumptions of Theorem \ref{4-dim} is equivalent to  (\ref{2-form}) in a coordinate-set $(p_1,y_1,y_2,y_3)$ on $\mathbb K^4$, since any contact form on $\mathbb K^3=\{p_1=0\}$ is equivalent to $dy_3+y_1dy_2$.

The set-germ $\Sigma_{22}=\{y\in \Sigma_2:\sigma|_y=0\}$ can be described as $$\{y\in \Sigma_2:a(y)=b(y)=0\}.$$ If  $\dim \text{span} \  j_0^1\sigma^{n-1}$ is $2$ then $\Sigma_{22}$ is a germ
of a smooth curve on $\Sigma_2$.

For $\mathbb K=\mathbb R$ if $(\frac{\partial b}{\partial y_2}(0))^2+\frac{\partial b}{\partial y_1}(0)\frac{\partial h}{\partial y_2}(0)$ is positive then $\omega$ has a  hyperbolic $\Sigma_{220}$ singularity, if it is negative then $\omega$
has an elliptic $\Sigma_{220}$ singularity and if it is zero then $\omega$
has a parabolic $\Sigma_{221}$  singularity \cite{Martinet}.  Roussarie has shown the stability of $\Sigma_{220}$ singularities \cite{Roussarie}. Golubitsky and Tischner have proved that $\Sigma_{221}$ singularity is not  stable \cite{G_T}.

 The normal forms of  $\Sigma_{220}$ singularities are  presented below
 \begin{eqnarray}
 &\text{hyperbolic} \  \Sigma_{220} :& \nonumber \\
 & d(p_1(dy_3+y_1dy_2)) +(dy_3+y_1dy_2)\wedge(y_1dy_1-y_2dy_2),&\nonumber \\
 & &\nonumber \\
 &\text{elliptic} \  \Sigma_{220} :& \nonumber \\
  & \ d(p_1(dy_3+y_1dy_2)) +(dy_3+y_1dy_2)\wedge(y_1dy_1+y_2dy_2).&\nonumber
\end{eqnarray}

\end{ex}


\section{Determination by the restriction of $\omega$ to $T\Sigma_2$ in dimension
4.}\label{deter4}

In \cite{4dim} we proved the following result on determination of the equivalence class of a $\mathbb C$-analytic singular symplectic form-germ $\omega$ by its restriction to the structurally smooth Martinet hypersurface. 

\begin{thm} \label{C-ana}
Let $\omega_0$ and $\omega_1$ be germs of  $\mathbb C$-analytic
singular symplectic forms on $\mathbb C^4$ with a common
structurally smooth Martinet hypersurface $\Sigma_2$ at $0$ and
$\text{rank}(\iota^{\ast}\omega_0|_{0})=\text{rank}(\iota^{\ast}\omega_1|_0) =0$.

If $\iota^{\ast}\omega_0=\iota^{\ast}\omega_1=\sigma$ and there
does not exist  a $\mathbb C$-analytic vector field-germ $X$
on $\Sigma_2$ at $0$ such that $X\rfloor \sigma=0$ and $X|_0\ne 0$
then there exists a $\mathbb C$-analytic diffeomorphism-germ
$\Psi :(\mathbb C^{4},0)\rightarrow (\mathbb C^{4},0)$ such that
\[
\Psi ^{\ast }\omega_1 =\omega_0.
\]
\end{thm}

In the analogous result in $\mathbb R$-analytic category (\cite{4dim}) the fixed canonical orientation of the Martinet hypersurface is needed ( see Example \ref{orientation} )

\begin{thm} \label{R-ana}
Let $\omega_0$ and $\omega_1$ be germs of $\mathbb R$-analytic
singular symplectic forms on $\mathbb R^4$ with a common
structurally smooth Martinet hypersurface $\Sigma_2$ at $0$ and
$\text{rank}(\iota^{\ast}\omega_0|_{0})=\text(rank\iota^{\ast}\omega_1|_0) =0$.

If $\iota^{\ast}\omega_0=\iota^{\ast}\omega_1=\sigma$, $\omega_0$
and $\omega_1$ define the same canonical orientation of $\Sigma_2$
and there does not exist  an $\mathbb R$-analytic vector
field-germ $X$ on $\Sigma_2$ at $0$ such that $X\rfloor \sigma=0$ and
$X|_0\ne 0$ then there exists  an $\mathbb R$-analytic
diffeomorphism-germ $\Psi :(\mathbb R^{4},0)\rightarrow (\mathbb
R^{4},0)$ such that
\[
\Psi ^{\ast }\omega_1 =\omega_0.
\]
\end{thm}

One can also find the normal form of a singular symplectic
form-germ on $\mathbb K^4$ at $0$ which does not satisfy the
assumptions of Theorems \ref{R-ana}, \ref{C-ana}. The following result is also
true in the smooth category (\cite{4dim}).

\begin{prop}\label{propC} Let $\omega$ be  a $\mathbb K$-analytic (smooth) singular symplectic form-germ
on $\mathbb K^4$ with a structurally smooth Martinet hypersurface
at $0$ and $\text{rank}(\iota^{\ast}\omega|_{0}) =0$.

If there exists a $\mathbb K$-analytic (smooth) vector field-germ $X$
on $\Sigma_2$ at $0$ such that $X\rfloor \sigma=0$ and $X|_0\ne 0$
then there exists of a $\mathbb K$-analytic (smooth) diffeomorphism-germ
$\Psi :(\mathbb K^{4},0)\rightarrow (\mathbb K^{4},0)$ such that
\[
\Psi ^{\ast }\omega =d(p_1(dx+C dy +zdy))+g(x,y)dx \wedge dy
\]
or
\[
\Psi ^{\ast }\omega =d(p_1(dy+C dx +zdx))+g(x,y)dx \wedge dy,
\]
where $C\in \mathbb K$ and $g$ is a $\mathbb K$-analytic
function-germ on $\mathbb K^4$ at $0$ that does not depend on
$p_1$ and $z$.
\end{prop}

In this section we find conditions for the determination of the equivalence class of a
smooth or $\mathbb R$-analytic singular symplectic form on $\mathbb R^{4}$ by its pullback to the Martinet
hypersurface only.

We need some notions from commutative algebra (see Appendix 1
of \cite{JZ1}, \cite{E}) to formulate the result in the smooth
category. We recall that a sequence of elements $a_1,\cdots, a_r$
of a proper ideal $I$ of a ring $R$ is called {\it regular} if
$a_1$ is a non-zero-divisor of $R$ and $a_i$ is a non-zero-divisor of
$R/<a_1,\cdots,a_{i-1}>$ for $i=2,\cdots,r$. Here
$<a_1,\cdots,a_i>$ denotes the ideal generated by $a_1,\cdots,
a_i$. The {\it length} of a regular sequence $a_1,\cdots, a_r$ is
$r$.

The {\it depth} of the proper ideal $I$ of the ring $R$ is the
supremum of lengths of regular sequences in $I$. We denote it by
$\text{depth}(I)$. If $I=R$ then we define $\text{depth}(I)=\infty$.

Let $\sigma$ be  a smooth ($\mathbb K$-analytic) closed
$2$-form-germ on $\Sigma_2=\mathbb K^3$ and $\text{rank} (\sigma|_0)=0$. In the
local coordinate system  $(x,y,z)$ on $\Sigma_2$ we have
$\sigma=ady\wedge dz+b dz\wedge dx +c dx\wedge dy$, where $a,b,c$
are smooth ($\mathbb K$-analytic) function-germs on $\Sigma_2$. By
$I(\sigma)$ we denote the ideal of the ring of smooth ($\mathbb
K$-analytic) function-germs on $\Sigma_2$ generated by $a,b,c$
i.e. $I(\sigma)=<a,b,c>$. It is easy to see that $I(\sigma)$ does
not depend on the local coordinate system  on $\Sigma_2$. $\sigma$
satisfies the condition $\alpha\wedge \sigma=0$, where $\alpha$ is
 a contact form-germ on $\mathbb K^3$. It implies that
$I(\sigma)$ is generated by two function-germs.

In the $\mathbb K$-analytic category if $\text{depth} I(\sigma)\ge 2$
then the two generators of $I(\sigma)$ form a regular sequence of
length 2 (see \cite{E}). One can easily check that it implies that
there does not exist  a $\mathbb K$-analytic vector field-germ
on $\Sigma_2$ such that $X\rfloor \sigma=0$ and $X|_0\ne 0$. The
inverse implication is not true in general. Now we can prove the following theorem.

\begin{thm} \label{smooth}
Let $\omega_0$ and $\omega_1$ be germs of smooth or $\mathbb R$-analytic singular
symplectic forms on $\mathbb R^4$ with a common structurally
smooth Martinet hypersurface $\Sigma_2$ at $0$ and
$\text{rank}(\iota^{\ast}\omega_0|_{0})=\text{rank}(\iota^{\ast}\omega_1|_0) =0$.

If $\iota^{\ast}\omega_0=\iota^{\ast}\omega_1=\sigma$
and the two generators of the ideal $I(\sigma)$ form a regular
sequence of length 2 then there exists a smooth or $\mathbb R$-analytic
diffeomorphism-germ $\Psi :(\mathbb R^{4},0)\rightarrow (\mathbb
R^{4},0)$ such that
\[
\Psi ^{\ast}\omega_1 =\omega_0.
\]
\end{thm}

\begin{proof}
By Theorem \ref{4-dim} (a) we obtain
$\omega_0=d(p_1\pi^{\ast}\alpha_0)+\sigma$ and
$\omega_1=d(p_1\pi^{\ast}\alpha_1)+\sigma$, where $\alpha_0, \
\alpha_1$ are germs of smooth contact forms on
$\Sigma_2=\{p_1=0\}$ such that
$\alpha_0\wedge\sigma=\alpha_1\wedge\sigma=0$.

$\alpha_0$ is a contact form therefore we can
find a coordinate system  $(x,y,z)$ on $\Sigma_2$  such that
$\alpha_0=dz+xdy$. Let
$\sigma=ady\wedge dz + b dz\wedge dx + c dx\wedge dy$, where $a$,
$b$, $c$ are function-germs on $\Sigma_2$ vanishing at $0$.
From $\alpha_0\wedge \sigma=0$ we get $c=-xb$. Thus $I(\sigma)=<a,b,c>=<a,b>$.
The $2$-form germ $\sigma$ is closed. It implies that $\frac{\partial a}{\partial x}+\frac{\partial b}{\partial y}-x\frac{\partial b}{\partial z}=0$. Thus we have
\begin{equation}\label{closed}
\frac{\partial a}{\partial x}(0)+\frac{\partial b}{\partial y}(0)=0
\end{equation}
Let $\alpha_1=f dx+g dy +h dz$, where $f, g, h$ are
functions-germs on $\Sigma_2$. From $\alpha_1\wedge \sigma=0$ we
obtain the equation
\begin{equation}
\label{s-rown}
af+b(g-xh)=0
\end{equation}
and $a(0)=b(0)=0$.

By assumptions $a,b$ is a regular sequence.

Therefore $f=rb$ and
$g-xh=-ra$, where $r$ is a smooth function-germ
on $\Sigma_2$ at $0$.

Thus $1$-form germ $\alpha_1$ has the following form
\begin{equation}\label{alfa}
\alpha_1=rb dx +(xh-ra) dy +hdz.
\end{equation}
 Thus $\alpha_1|_0=h(0)dz$ since $a(0)=b(0)=0$ and $h(0)\ne 0$, because $\alpha_1$ is a contact form-germ. It implies that
\begin{equation}\label{kernels}
\ker \alpha_0|_0=\ker \alpha_1|_0.
\end{equation}
By (\ref{alfa}) we get $$\alpha_1\wedge d\alpha_1|_0=\left((h(0))^2-h(0)r(0)\left(\frac{\partial a}{\partial x}(0)+\frac{\partial b}{\partial y}(0)\right)\right)dx\wedge dy\wedge dz.$$

By (\ref{closed}) we obtain that
$$\alpha_1\wedge d\alpha_1|_0=(h(0))^2 dx\wedge dy\wedge dz, \ \ \alpha_0\wedge d\alpha_0|_0=dx\wedge dy \wedge dz.$$ Since $h(0)\ne 0$ both $3$-forms define the same orientation of $\Sigma_2$.
Therefore from (\ref{kernels}) we finish the proof by Theorem \ref{4-dim} (b).
\end{proof}

\begin{ex}\label{orientation}
Let $\omega$ be a closed $2$-form-germ on $\mathbb R^4$ in coordinates $(p_1,x,y,z)$ of the following form $d(p_1\alpha)+\sigma$, where
$$\alpha=dz+xdy,  \  \ \sigma=x(dz+xdy)\wedge (a(x,y,z) dx-b(x)dy),$$
 $a(x,y,z)=a_1x+a_2y+a_3z$ and  $b(x)=\frac{a_3}{3}x^2-\frac{a_2}{2}x$.

 It is easy to check that, $d\omega=0$,  $\Sigma_2(\omega)=\{p_1=0\}$, $\alpha$ is contact form-germ on $\{p_1=0\}$, $\omega_{T\Sigma_2}=\sigma$  and $\alpha\wedge \sigma=0$.

 Let $\omega_1$ be a closed $2$-form-germ on $\mathbb R^4$ of the following form
 $$d(p_1(h(x,y,z)\alpha+r(x,y,z)(a(x,y,z)dx-b(x)dy)))+\sigma,$$
 where $h$ and $r$ are $\mathbb R$-analytic function-germs on $\{p_1=0\}$ and $h(0)r(0)\ne 0$. It is easy to check that
 $d\omega_1=0$,  $\Sigma_2(\omega_1)=\{p_1=0\}$,  $\omega_1|_{T\Sigma_2}=\sigma$  and
 $$\left(h(x,y,z)\alpha+r(x,y,z)(a(x,y,z)dx-b(x)dy)\right)\wedge \sigma=0.$$
 The $1$-form-germ $h(x,y,z)\alpha+r(x,y,z)(a(x,y,z)dx-b(x)dy)$ is a contact form-germ on $\{p_1=0\}$ iff   $ h(0)(h(0)-1/2a_2r(0))\ne 0$.

 Thus $\omega$ and $\omega_1$  are two singular symplectic form-germs with the same restriction $\sigma$ to the common Martinet hypersurface $\{p_1=0\}$. But the canonical orientations of the Martinet hypersurface defined by $\omega$ and $\omega_1$ are different if $ h(0)(h(0)-1/2a_2r(0))<0$.
 \end{ex}

\section{The complete set of invariants for singular symplectic
forms with singular Martinet hypersurfaces.}\label{singular}

In this section we consider singular symplectic forms with singular Martinet hypersurfaces. For any smooth ( $\mathbb K$-analytic) function $f$ on $\mathbb K^{2n}$ there exists closed $2$-form $\omega$ such that $\Sigma_2(\omega)$ is $f^{-1}(0)$. Such singular symplectic form can be constructed in the following way   (see \cite{Geometry})
$$\omega=d(\frac{1}{n!}\int_0^{x_1}f(t,x_2,\cdots,x_{2n})dt dx_2+\sum_{i=2}^{n}x_{2i-1}dx_{2i}),$$
where  $(x_1,\cdots,x_{2n})$ is the coordinate system  on $\mathbb K^{2n}$. Then $\omega^n=f(x)dx_1\cdots \wedge dx_{2n}$.

We assume that the Martinet hypersurface  is a quasi-homogeneous hypersurface with an isolated singularity.
Under these assumptions we can prove that the equivalence class of a singular symplectic form is determine by its restriction to the regular part of the singular Martinet hypersurface and its canonical orientation.

First we recall the notion of quasi-homogeneity and its properties.

\begin{defn}
\label{def-qh} The germ at $0$ of a set $N\subset \mathbb K^m$ is
called quasi-homogeneous if there exist a local coordinate system 
$(x_1,\ldots,x_m)$ and positive integers $\lambda _1,\ldots,\lambda _m$
such that the following holds: if a point with coordinates
$(x_1,\cdots,x_m)$ belongs to $N$ then for any $t\in [0;1]$ the point with
coordinates $(t^{\lambda_1}x_1,\cdots,t^{\lambda_m}x_m)$  also belongs to $N$.

A function-germ $f$ at $0$ on $\mathbb K^m$ is quasi-homogeneous if there exist a local coordinate system 
$(x_1,\ldots,x_m)$ and positive integers $\lambda _1,\ldots,\lambda _m, \delta$ such that $ f(t^{\lambda_1}x_1,\cdots,t^{\lambda_m}x_m)=t^{\delta}f(x_1,\ldots,x_m)$ for any $t\in [0;1]$ and any $(x_1,\ldots,x_m)$.
\end{defn}

It is obvious that if a function-germ $f$ on $\mathbb K^m$ is quasi-homogeneous then $f^{-1}(0)$ is a quasi-homogeneous subset-germ of $\mathbb K^m$.
The following property of quasi-homogeneous subset-germs is crucial for our study.

\begin{thm}[\cite{Re} in $\mathbb C$-analytic category, \cite{DJZ1} in $\mathbb R$-analytic and smooth categories]
\label{Poincare}
If $N$ is a quasi-homogeneous subset-germ of $\mathbb K^m$ then any closed $k$-form-germ vanishing at every point of $N$ is a differential of a $(k-1)$-form-germ vanishing at every point of $N$.
\end{thm}

To prove our result we also need the following division property.
\begin{defn}
\label{def-divis}
A differential $1$-form-germ $\alpha$ on $\mathbb K^m$ has $k$-division property if for any differential $k$-form-germ $\beta$ such that $\alpha\wedge \beta=0$ there exists  a differential $(k-1)$-form-germ $\gamma$ such that $\beta=\alpha\wedge \gamma$.
\end{defn}

Let $\mathcal O$ denotes the ring of $\mathbb K$-analytic or smooth function-germs at $0$ and let $f\in \mathcal O$. We recall the definition of an isolated singularity.

\begin{defn}
A singular hypersurface-germ  $\{f=0\}$ has an isolated singularity at $0$ if
$$
\dim_{\mathbb K} \frac{\mathcal O}{<\frac{\partial f}{\partial x_1}, \cdots, \frac{\partial f}{\partial x_m}>}<\infty.
$$
\end{defn}

The differential of a function-germ with an isolated singularity has the division property.

\begin{thm}[\cite{Moussu}]
\label{divis}
If $\{f=0\}$ has an isolated singularity at $0$ then $df$ has $k$-division property for $k=1,\cdots,m-1$.
\end{thm}

Now we are ready to prove the main result of this section.

\begin{thm}\label{sing}
Let $\omega_0$ and $\omega_1$ be germs of  smooth ($\mathbb
K$-analytic) singular symplectic forms on $\mathbb K^{2n}$ with a
common singular Martinet hypersurface $\Sigma_2$ at $0$. Let
$\Sigma_2$ be a quasi-homogeneous hypersurface-germ with
an isolated singularity at $0$.

If $\omega_0$ and $\omega_1$ have the same restriction to the
regular part of $\Sigma_2$ and $\omega_0$, $\omega_1$ define the
same canonical orientation of the regular part of $\Sigma_2$  then
there exists  a smooth ($\mathbb K$-analytic)
diffeomorphism-germ $\Psi :(\mathbb K^{2n},0)\rightarrow (\mathbb
K^{2n},0)$ such that
\[
\Psi ^{\ast }\omega_1 =\omega_0.
\]
\end{thm}

\begin{proof} We may find
a coordinate system  such that $\omega_0^n=f\Omega$, where $f$ is a
quasi-homogeneous function-germ with an isolated singularity at $0$
and $\Omega$ is  a volume form-germ on $\mathbb K^{2n}$ . Thus $\omega_1^n=gf\Omega$, where $g$ is a
function-germ, such that $g(0)>0$, because $\Sigma_2=\Sigma_2(\omega_0)=\Sigma_2(\omega_1)$, $\omega_0$ and
$\omega_1$ define the same orientation of the regular part of $\Sigma_2$. The singular symplectic form-germs $\omega_0$
and $\omega_1$ have the same restriction to the regular part of
$\Sigma_2$. Thus there exists  a $3$-form-germ $\beta$ such
that
\begin{equation}\label{restr}
df\wedge(\omega_1-\omega_0)=f\beta.
\end{equation}
Multiplying both sides of the above formula by $df\wedge$ we
obtain $ fdf\wedge\beta=0. $ But $\Sigma_2$ is nowhere dense thus
this implies that $ df\wedge\beta=0$. The hypersurface-germ $\{f=0\}$ has an isolated
singularity at $0$, therefore by Theorem \ref{divis} $df$ has $k$-division property  for $k=1,\cdots, 2n-1$. Thus we obtain $
\beta=df\wedge\gamma, $ where $\gamma$ is a $2$-form-germ.
From the above formula and (\ref{restr}) we obtain $
df\wedge(\omega_1-\omega_0-f\gamma)=0. $ By $2$-division property
of $df$  we get that
\begin{equation}\label{algrest}
\omega_1-\omega_0=f\gamma+df\wedge \delta,
\end{equation}
where $\delta$ is a $1$ form-germ.

The $2$-form-germ $\omega_1-\omega_0=f(\gamma- d\delta)+d(f\delta)$ is closed. It implies that the $2$-form $f(\gamma- d\delta)$ is closed too and it vanishes at every point of $\Sigma_2=\{f=0\}$.
Since $\Sigma_2$ is quasi-homogeneous by Theorem \ref{Poincare} we obtain that there exists a $1$ form-germ
$\alpha$ such that
\begin{equation}\label{PLP}
\omega_1-\omega_0=d(f\alpha)
\end{equation}
Now we use Moser's homotopy  method (\cite{Moser}). Let
$$
\omega_t=\omega_0+t(\omega_1-\omega_0)=\omega_0+td(f\alpha),
$$
for $t\in [0;1]$. We look for germs of diffeomorphisms $\Phi_t$
such that
\begin{equation}
\label{singhom} \Phi_t^{\ast}\omega_t=\omega_0,\  \text{for} \ t \in
[0;1], \ \Phi_0=Id.
\end{equation}
 Differentiating the above homotopy equation
by $t$, we obtain
\[
d(V_t\rfloor\omega_t)=d(f\alpha),
\]
where $V_t=\frac{d}{dt}{\Phi_t}$. Therefore we have to solve the
following equation
\begin{equation}\label{singrl}
V_t\rfloor\omega_t=f\alpha.
\end{equation}
First we calculate $\Sigma_2(\omega_t)$. It is easy to see that
$$
\omega_1^n=(\omega_0+d(f\alpha))^n=\omega_0^n+n(fd\alpha+df\wedge\alpha)\wedge
\omega_0^{n-1}+f\kappa,
$$
where $\kappa$ is a $2n$-form-germ such that $\kappa|_0=0$
(because $df|_0=0$). But
$\Sigma_2(\omega_0)=\Sigma_2(\omega_1)=\{f=0\}$. Thus if we
restrict both sides of the above formula to $\{f=0\}$ we obtain
that $df\wedge\alpha\wedge \omega_0^{n-1}|_{\{f=0\}}=0$. Hence
there exists a function-germ $h$ such that
\begin{equation}\label{h}
 df\wedge\alpha\wedge
\omega_0^{n-1}=hf\Omega.
\end{equation}
But $\omega_1^n=gf\Omega$. Thus we obtain that
\begin{equation}\label{g}
g(0)=1+n\left(\frac{d\alpha\wedge
\omega_0^{n-1}}{\Omega}|_0+h(0)\right).
\end{equation}

No we calculate
$$
\omega_t^n=(\omega_0+td(f\alpha))^n=\omega_0^n+n(fd\alpha+df\wedge\alpha)\wedge
\omega_0^{n-1}t+f\kappa_t=fg_t\Omega,
$$
where $\kappa_t$ is a $2n$-form-germ such that
$\kappa_t|_0=0$ for $t\in [0;1]$ and $g_t$ is a function-germ.
Thus
$$
g_t(0)=1+tn\left(\frac{d\alpha\wedge
\omega_0^{n-1}}{\Omega}|_0+h(0)\right).
$$
From (\ref{g}) we obtain that $g_t(0)=1+t(g(0)-1)$. But $g(0)>0$,
therefore $g_t(0)>0$ for $t\in[0;1]$. Thus
$\Sigma_2(\omega_t)=\{f=0\}$ and $\omega_t$ define the same
orientation of $\Sigma_2$ for any $t$.

Because $\{f=0\}$ is nowhere dense, equation (\ref{singrl}) is
equivalent to
\[
V_t\rfloor\omega_t^n=nf\alpha\wedge\omega_t^{n-1}
\]
and $\omega_t^n=fg_t\Omega$. Therefore we have to solve the
following equation
\begin{equation}
\label{sing-n2-linear} V_t\rfloor g_t\Omega
=n\alpha\wedge\omega_t^{n-1}.
\end{equation}
Now we prove that the right hand side of (\ref{sing-n2-linear})
vanishes at $0$.
 It is easy to see that
 \begin{equation}\label{zeroatzero}
 \alpha\wedge\omega_t^{n-1}|_0=\alpha\wedge\omega_0^{n-1}|_0.
 \end{equation}
 The function-germ $f$ is quasi-homogeneous. Let $E$ be the Euler vector field for
 $f$ i.e. $E\rfloor df=f$ and $E|_0=0$ (see \cite{DR}).
 From (\ref{h}) we get that $$
 df\wedge\alpha\wedge
\omega_0^{n-1}=hf\Omega.
$$
Thus
$$
df\wedge\alpha\wedge \omega_0^{n-1}=h(E\rfloor df)\Omega=df\wedge
(hE\rfloor \Omega),
$$
because
$$
(E\rfloor df)\Omega=df\wedge E\rfloor \Omega.
$$
Hence
$$
df\wedge(\alpha\wedge \omega_0^{n-1}-hE\rfloor \Omega)=0.
$$
By
$(2n-1)$-division property of $df$ we get that
$$
\alpha\wedge \omega_0^{n-1}-hE\rfloor \Omega=df\wedge \theta,
$$
where $\theta$ is a $(2n-2)$-form-germ. From
(\ref{zeroatzero}) we get
$$
\alpha\wedge\omega_t^{n-1}|_0=0,
$$
because $E|_0=0$ and $df|_0=0$.
 Hence
we can find a smooth solution $V_t$ of (\ref{sing-n2-linear}) such
that $V_t|_0=0$. Therefore there exit germs of diffeomorphisms
$\Phi_t$, which satisfy (\ref{singhom}). For $t=1$ we have
$\Phi_1^{\star}\omega_1=\omega_0$.
\end{proof}


\end{document}